\documentclass[12pt]{amsart}
\usepackage{amscd,amssymb}
\usepackage{graphicx}
\usepackage{caption}
\def\F{{\mathcal F}}

\def\frk{\frak}               

\def\Phi{{\frk n}}
\def\Phi{{\frk N}}
\def\opn#1#2{\def#1{\operatorname{#2}}} 
\opn\chara{char} \opn\length{\ell} \opn\pd{pd} \opn\rk{rk}
\opn\projdim{proj\,dim} \opn\injdim{inj\,dim} \opn\rank{rank}
\opn\depth{depth} \opn\grade{grade} \opn\height{height}
\opn\embdim{emb\,dim} \opn\codim{codim}

\opn\Tr{Tr} \opn\bigrank{big\,rank}
\opn\superheight{superheight}\opn\lcm{lcm}
\opn\trdeg{tr\,deg}
\opn\reg{reg} \opn\lreg{lreg} \opn\ini{in} \opn\lpd{lpd}
\opn\size{size}
\opn\div{div} \opn\Div{Div} \opn\cl{cl} \opn\Cl{Cl}
\opn\Spec{Spec} \opn\Supp{Supp} \opn\supp{supp} \opn\Sing{Sing}
\opn\Ass{Ass} \opn\Min{Min}
\opn\Ann{Ann} \opn\Rad{Rad} \opn\Soc{Soc}
\opn\Im{Im} \opn\Ker{Ker} \opn\Coker{Coker} \opn\Am{Am}
\opn\Hom{Hom} \opn\Tor{Tor} \opn\Ext{Ext} \opn\End{End}
\opn\Aut{Aut} \opn\id{id}

\opn\nat{nat}
\opn\pff{pf}
\opn\Pf{Pf} \opn\GL{GL} \opn\SL{SL} \opn\mod{mod} \opn\ord{ord}
\opn\Gin{Gin} \opn\Hilb{Hilb}
\opn\aff{aff} \opn\con{conv} \opn\relint{relint} \opn\st{st}
\opn\lk{lk} \opn\cn{cn} \opn\core{core} \opn\vol{vol}
\opn\link{link} \opn\star{star}
\opn\gr{gr}

\def\pot#1#2{#1[\kern-0.28ex[#2]\kern-0.28ex]}

\opn\dirlim{\underrightarrow{\lim}}
\opn\inivlim{\underleftarrow{\lim}}
\def\Implies{\ifmmode\Longrightarrow \else
        \unskip${}\Longrightarrow{}$\ignorespaces\fi}
\def\implies{\ifmmode\Rightarrow \else
        \unskip${}\Rightarrow{}$\ignorespaces\fi}
\def\iff{\ifmmode\Longleftrightarrow \else
        \unskip${}\Longleftrightarrow{}$\ignorespaces\fi}

\let\:=\colon
\newtheorem{Theorem}{Theorem}[section]
\newtheorem{Lemma}[Theorem]{Lemma}

\newtheorem{Remark}[Theorem]{Remark}

\newtheorem{Example}[Theorem]{Example}

\newtheorem{Definition}[Theorem]{Definition}

\opn\Syz{Syz} \opn\Im{Im} \opn\Ker{Ker} \opn\Coker{Coker}
\opn\Am{Am} \opn\Hom{Hom} \opn\Tor{Tor} \opn\Ext{Ext} \opn\End{End}
\opn\Aut{Aut} \opn\id{id}

\opn\nat{nat}
\opn\pff{pf}
\opn\Pf{Pf} \opn\GL{GL} \opn\SL{SL} \opn\mod{mod} \opn\ord{ord}
\opn\Gin{Gin}\opn\min{min}
\opn\Hilb{Hilb}\opn\adeg{adeg}\opn\std{std}\opn\ip{infpt}
\opn\Pol{Pol}\opn\sdepth{sdepth}\opn\infpt{infpt}
\opn\depth{depth}\opn\sqdepth{sqdepth}\opn{\Mon}{Mon}

\let\epsilon\varepsilon
\let\phi=\varphi
\let\kappa=\varkappa
\def\qed{\ifhmode\textqed\fi
      \ifmmode\ifinner\quad\qedsymbol\else\dispqed\fi\fi}
\def\textqed{\unskip\nobreak\penalty50
       \hskip2em\hbox{}\nobreak\hfil\qedsymbol
       \parfillskip=0pt \finalhyphendemerits=0}
\def\dispqed{\rlap{\qquad\qedsymbol}}

\opn\dis{dis}
\def\pnt{{\raise0.5mm\hbox{\large\bf.}}}

\opn\Lex{Lex}

\begin{document}
\title
{Topological and Algebraic Characterizations of Gallai-Simplicial
Complexes}

\author{Imran Ahmed, Shahid Muhmood}

\address{COMSATS Institute of Information Technology, Lahore, Pakistan}
\email{drimranahmed@ciitlahore.edu.pk}
\address{COMSATS Institute of Information Technology, Lahore, Pakistan}
\email{shahid\_nankana@yahoo.com}

 \maketitle
\begin{abstract} We recall first Gallai-simplicial
complex $\Delta_{\Gamma}(G)$ associated to Gallai graph $\Gamma(G)$
of a planar graph $G$, see \cite{AKNS}. The Euler characteristic is
a very useful topological and homotopic invariant to classify
surfaces. In Theorems \ref{t1} and \ref{t3}, we compute Euler
characteristics of Gallai-simplicial complexes associated to
triangular ladder and prism graphs, respectively. 

Let $G$ be a finite simple graph on $n$ vertices of the form $n=3l+2$ or $3l+3$. In Theorem
\ref{t5}, we prove that $G$ will be $f$-Gallai graph for the following types of constructions of $G$.\\
{\bf Type 1.}  When $n=3l+2$. $G=\mathbb{S}_{4l}$ is a graph
consisting of two copies of star graphs $S_{2l}$ and $S'_{2l}$ with
$l\geq 2$ having $l$ common vertices.\\ {\bf Type 2.} When $n=3l+3$.
$G=\mathbb{S}_{4l+1}$ is a graph consisting of two star graphs
$S_{2l}$ and $S_{2l+1}$ with $l\geq 2$ having $l$ common vertices.
 \vskip 0.4 true cm
 \noindent
\end{abstract}

\pagestyle{myheadings} \markboth{\centerline {\scriptsize Ahmed and
Muhmood}}
         {\centerline {\scriptsize Topological and Algebraic Characterizations of Gallai-Simplicial
Complexes}}

\maketitle

\section{Introduction}



Let $X$ be a finite CW complex of dimension $N$. The Euler
characteristic is a function $\chi$ which associates to each $X$ an
integer $\chi(X)$. More explicitly, the Euler characteristic of $X$
is defined as the alternating sum
$$\chi(X)=\sum\limits_{k=0}^{N}(-1)^k\beta_k(X)$$
with $\beta_k(X)=rank(H_k(X))$ the $k$-th Betti number of $X$.

The Euler characteristic is a very useful topological and homotopic
invariant to classify surfaces. The Euler characteristic is uniquely
determined by excision $\chi(X)=\chi(C)+\chi(X\backslash C)$, for
every closed subset $C\subset X$. The excision property has a dual
form $\chi(X)=\chi(U)+\chi(X\backslash U)$, for every open subset
$U\subset X$, see \cite{H} and \cite{M} for more details.


We consider a planar graph $G$, the Gallai graph $\Gamma(G)$ of $G$
is a graph having edges of $G$ as its vertices, that is,
$V(\Gamma(G))=E(G)$ and two distinct edges of $G$ are adjacent in
$\Gamma(G)$ if they are adjacent in $G$ but dot span a triangle. The
buildup of the 2-dimensional Gallai-simplicial complex
$\Delta_{\Gamma}(G)$ from a planar graph $G$ is an abstract idea
similar to building an origami shape from a plane sheet of paper by
defining a crease pattern, see \cite{AKNS}.

Let $S=k[x_1,\ldots,x_n]$ be a polynomial ring over an infinite
field $k$. There is a natural bijection between a square free
monomial ideal and a simplicial complex written as
$\Delta\leftrightarrow I_{\mathcal{N}}(\Delta)$, where
$I_{\mathcal{N}}(\Delta)$ is the Stanley-Reisner ideal or non-face
ideal of $\Delta$, see for instance \cite{WBH}. In \cite{SF}, Faridi
introduced another correspondence $\Delta\leftrightarrow
I_{\mathcal{F}}(\Delta)$, where $I_{\mathcal{F}}(\Delta)$ is the
facet ideal of $\Delta$. She discussed its connections with the
theory of Stanley-Reisner rings.

In \cite{AAAB} and \cite{AMBZ}, the authors investigated the
correspondence $\delta_{\mathcal{F}}(I)\leftrightarrow
I\leftrightarrow \delta_{\mathcal{N}}(I)$, where
$\delta_{\mathcal{F}}(I)$ and $\delta_{\mathcal{N}}(I)$ are facet
and non-face simplicial complexes associated to the square free
monomial ideal $I$ (respectively). A square free monomial ideal $I$
in $S$ is said to be an $f$-ideal if and only if both
$\delta_{\mathcal{F}}(I)$ and $\delta_{\mathcal{N}}(I)$ have the
same $f$-vector. The concepts of $f$- ideals is important in the
sense that it discovers new connections between both the theories.
The complete characterization of $f$-ideals in the polynomial ring
$S$ over a field $k$  can be found in \cite{AMBZ}. A simple finite
graph $G$ is said to be the $f$-graph if its edge ideal $I(G)$ is an
$f$-ideal of degree 2, see \cite{MAZ}.

In Theorems \ref{t1} and \ref{t3}, we compute Euler characteristics
of Gallai-simplicial complexes associated to triangular ladder and
prism graphs, respectively.

Let $G$ be a finite simple graph on $n$ vertices of the form $n=3l+2$ or $3l+3$. In Theorem
\ref{t5}, we prove that $G$ will be $f$-Gallai graph for the following types of constructions of $G$.\\
{\bf Type 1.}  When $n=3l+2$. $G=\mathbb{S}_{4l}$ is a graph
consisting of two copies of star graphs $S_{2l}$ and $S'_{2l}$ with
$l\geq 2$ having $l$ common vertices.\\ {\bf Type 2.} When $n=3l+3$.
$G=\mathbb{S}_{4l+1}$ is a graph consisting of two star graphs
$S_{2l}$ and $S_{2l+1}$ with $l\geq 2$ having $l$ common vertices.

\section{Preliminaries}

\noindent A simplicial complex $\Delta$ on $[{n}]=\{1,\ldots, n\}$
is a collection of subsets of $[n]$ with the property that $\{i\}\in
\Delta$ for all $i$, and if $F\in \Delta$ then every subset of $F$
will belong to $\Delta$ (including empty set). The elements of
$\Delta$ are called faces of $\Delta$ and the dimension of a face
$F\in\Delta$ is defined as $|F|-1$, where $|F|$ is the number of
vertices of $F$. The faces of dimensions $0$ and $1$ are called
vertices and edges, respectively, and $dim\ \emptyset= -1$.

\noindent The maximal faces of $\Delta$ under inclusion are called
facets. The dimension of $\Delta$ is denoted by $\dim\Delta$ and is
defined as:
$$\dim\Delta=max\{\dim F\ |\ F\in\Delta\}.$$ A simplicial complex is
said to be pure if it has all the facets of the same dimension. If
$\{F_1,\ldots ,F_q\}$ is the set of all the facets of $\Delta$, then
$\Delta=<F_1,\ldots ,F_q>$.

We denote by $\Delta_n$ the closed $n$-dimensional simplex. Every
simplex $\Delta_n$ is homotopic to a point and thus
$$\chi(\Delta_n)=1,\,\forall\, n\geq 0.$$
Note that $\partial \Delta_n$ is homeomorphic to the $(n-1)$-sphere
$S^{n-1}$. Since $S^0$ is a union of two points, we have
$\chi(S^0)=2$. In general, the $n$-dimensional sphere is a union of
two closed hemispheres intersecting along the Equator which is a
$(n-1)$ sphere. Therefore,
$$\chi(S^n)=2\chi(\Delta_n)-\chi(S^{n-1})=2-\chi(S^{n-1}).$$
We deduce inductively
$$2=\chi(S^n)+\chi(S^{n-1})=\ldots=\chi(S^1)+\chi(S^0)$$
so that $\chi(S^n)=1+(-1)^n$. Now, note that the interior of
$\Delta_n$ is homeomorphic to $\mathbb{R}^n$ so that
$$\chi(\mathbb{R}^n)=\chi(\Delta_n)-\chi(\partial\Delta_n)=1-\chi(S^{n-1})=(-1)^n.$$

The excision property implies the following useful formula. Suppose
$$\emptyset\subset\Delta^{(0)}\subset\ldots\subset\Delta^{(N)}=\Delta$$
is an increasing filtration of $\Delta$ by closed subsets. Then,
$$\chi(\Delta)=\chi(\Delta^{(0)})+\chi(\Delta^{(1)}\backslash\Delta^{(0)})
+\ldots+\chi(\Delta^{(N)}\backslash\Delta^{(N-1)}).$$ We denote by
$\Delta^{(k)}$ the union of the simplices of dimension $\leq k$.
Then, $\Delta^{(k)}\backslash\Delta^{(k-1)}$ is the union of
interiors of the $k$-dimensional simplices. We denote by
$f_k(\Delta)$ the number of such simplices. Each of them is
homeomorphic to $\mathbb{R}^k$ and thus its Euler characteristic is
equal to $(-1)^k$. Consequently, the Euler characteristic of
$\Delta$ is given by
$$\chi(\Delta)=\sum\limits_{k=0}^N(-1)^k{f_k}(\Delta),$$
see \cite{H} and \cite{M}.

Let $\Delta$ be a simplicial complex of dimension $N$, we define its
$f$-vector by a $(N+1)$-tuple $f=(f_0,\ldots, f_N)$, where $f_i$ is
the number of $i$-dimensional faces of $\Delta$.

The following definitions serve as a bridge between the
combinatorial and algebraic properties of the simplicial complexes
over the finite set of vertices $[n]$.

Let $\Delta$ be a simplicial complex over the vertex set
$\{v_1,\ldots,v_n\}$ and $S=k[x_1,\ldots,x_n]$ be the polynomial
ring on $n$ variables.  We define the facet ideal of $\Delta$ by
$I_\mathcal{F}(\Delta)$, which is an ideal of $S$ generated by
square free monomials $x_{i_1}\ldots x_{i_s}$ where
$\{v_{i_1},\ldots,v_{i_s}\}$ is a facet of $\Delta$. We define the
non-face ideal or the Stanley-Reisner ideal of $\Delta$ by
$I_\mathcal{N}(\Delta)$, which is an ideal of $S$ generated by
square free monomials $x_{i_1}\ldots x_{i_s}$ where
$\{v_{i_1},\ldots,v_{i_s}\}$ is a non-face of $\Delta$.

Let $I=(M_1,\dots,M_q)$ be a square free monomial ideal in the
polynomial ring $S=k[x_1,\ldots,x_n]$, where $\{M_1,\dots,M_q\}$ is
a minimal generating set of $I$. We define a simplicial complex
$\delta_{\mathcal{F}}(I)$ over a set of vertices
$v_{1},\ldots,v_{n}$ with facets $F_1,\dots,F_q$, where for each
$i$, $F_i=\{v_j\ | \ x_j|M_i, 1\leq j\leq n\}$.
$\delta_{\mathcal{F}}(I)$ is said to be the facet complex of $I$. We
define a simplicial complex $\delta_{\mathcal{N}}(I)$ over a set of
vertices $v_{1},\ldots,v_{n}$, where $\{v_{i_1},\ldots,v_{i_s}\}$ a
face of $\delta_{\mathcal{N}}(I)$ if and only if the product
$x_{i_1}\ldots x_{i_s}$ does not belong to $I$. We call
$\delta_{\mathcal{N}}(I)$ the non-face complex or the
Stanley-Reisner complex of $I$.

To proceed further, we define the Gallai-graph $\Gamma(G)$, which is
a nice combinatorial buildup, see \cite{TG} and \cite{VBL}.

\begin{Definition} {\rm Let $G$ be a graph and $\Gamma(G)$ is said to be the Gallai graph of $G$ if
the following conditions hold;\\
1. Each edge of $G$ represents a vertex of $\Gamma(G)$.\\
2. If two edges are adjacent in $G$ that do not span a triangle in
$G$ then their corresponding vertices will be adjacent in
$\Gamma(G)$.}
\end{Definition}
\begin{Example} {\rm The graph $G$ and its
Gallai graph $\Gamma(G)$ are given in figures (i) and (ii),
repectively}.
\end{Example}

        \centerline{\includegraphics[width=14.0cm]{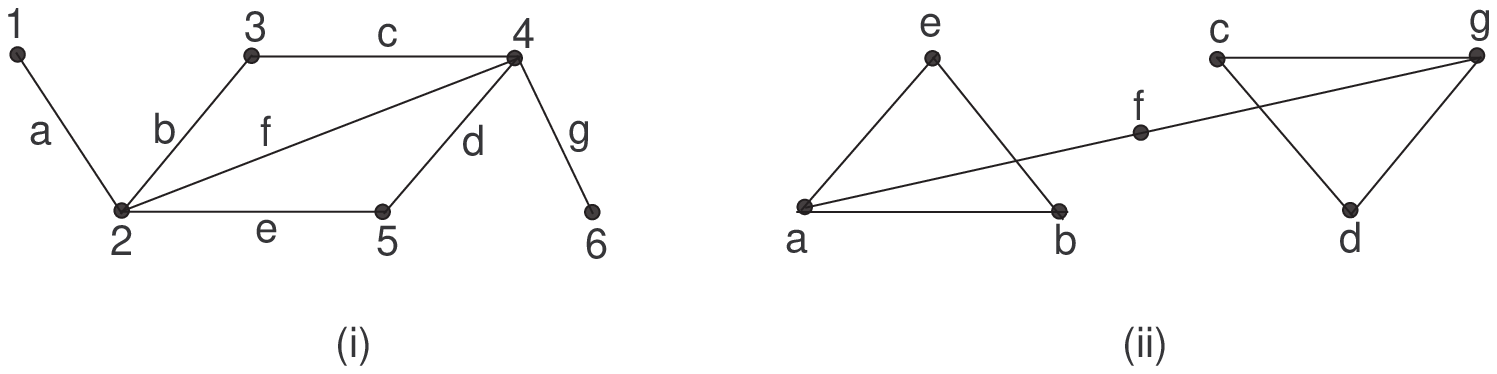}}
        \label{fig Ex5}

To define Gallai-simplicial complex $\Delta_\Gamma(G)$ of a planar
graph $G$, we introduce first a few notions, see \cite{AKNS}.

\begin{Definition}\cite{AKNS} {\rm Let $G$ be a  finite simple graph with vertex set
$V(G)=[n]$ and edge set $E(G)=\{e_{i,j}=\{i,j\}\,|\, i,j\in V(G)\}$.
We define the set of Gallai-indices $\Omega(G)$ of the graph $G$ as
the collection of subsets of $V(G)$ such that if $e_{i,j}$ and
$e_{j,k}$ are adjacent in $\Gamma(G)$, then $F_{i,j,k}=\{i,j,k\}\in
\Omega(G)$ or if $e_{i,j}$ is an isolated vertex in $\Gamma(G)$ then
$F_{i,j}=\{i,j\}\in \Omega(G)$.}
\end{Definition}

\begin{Definition}\cite{AKNS} {\rm A Gallai-simplicial complex $\Delta_\Gamma(G)$ of $G$ is
a simplicial complex defined over $V(G)$ such that
$$\Delta_\Gamma(G)=<F \ | \ F\in \Omega(G)>,$$ where $\Omega(G)$ is
the set of Gallai-indices of $G$.}
\end{Definition}
\begin{Example}
{\rm For the graph $G$ shown in figure below, its Gallai-simplicial
complex $\Delta_{\Gamma}(G)$ is given by}
$$\Delta_{\Gamma}(G)=<\{2,4\}, \{1,2,3\},\{1,3,4\},\{1,2,5\},\{1,4,5\}>.$$
\end{Example}
\begin{figure}[htbp]
        \centerline{\includegraphics[width=10.0cm]{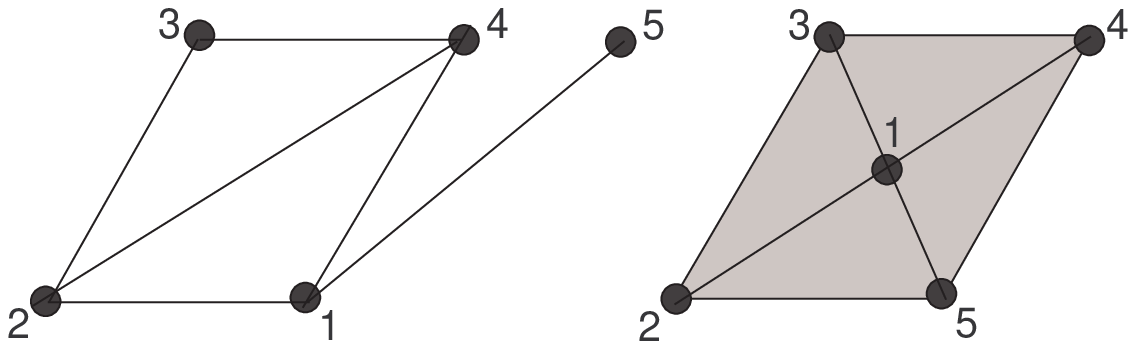}}
$G$ \ \ \ \ \ \ \ \ \ \ \ \ \ \ \ \ \ \ \ \ \ \ \ \ \ \ \ \ \ \ \ \
\ \ \ \ \ \  $\Delta_\Gamma(G)$
\end{figure}


\section{Characterizations of Gallai-Simplicial Complexes}
The ladder graph $L_{n}$ is a planar undirected graph with $2n$
vertices and $3n-2$ edges. The ladder graph $L_{n}$ is the cartesian
product of two path graphs $P_n$ and $P_2$, that is $L_n=P_n\times
P_2$ and looks like a ladder with $n$ rungs. The path graph $P_n$ is
a graph whose vertices can be listed in an order $v_1,\ldots,v_n$
such that $\{v_i,v_{i+1}\}$ is an edge for $1\leq i\leq n-1$. If we
add a cross edge between every two consecutive rungs of the ladder
then the resulting graph is said to be a triangular ladder graph
$L^\ast_n$ with $2n$ vertices and $4n-3$ edges.
\begin{Lemma}\label{l1}
Let $L^\ast_n$ be the triangular ladder graph on $2n$ vertices with
fixing the label of the edge-set $E(L^\ast_n)$ as follows;
$$E(L^\ast_n)=\{e_{1,2}, e_{2,3},\ldots,e_{2n-1,2n},e_{1,2n-1},e_{1,2n},\ldots,e_{n-1,n+1},e_{n-1,n+2}\}.$$
Then, we have\\
$\Omega(L^\ast_n)=\{F_{1,2,3},\ldots,F_{n-2,n-1,n},F_{n,n+1,n+2},\ldots,F_{2n-2,2n-1,2n},\\
F_{1,2,2n},F_{2,3,2n-1},\ldots,F_{n-1,n,n+2},F_{1,2,2n-2},\ldots,F_{n-2,n-1,n+1},\\
F_{1,2n-2,2n-1},\ldots,F_{n-2,n+1,n+2},F_{2,2n-1,2n},\ldots,F_{n-1,n+2,n+3}\}.$
\end{Lemma}
\begin{proof}
By definition, it is clear that $F_{i,i+1,i+2}\in \Omega(L^\ast_n)$
because $i$, $i+1$, $i+2$ are consecutive vertices of $2n$-cycle and
edges $e_{i,i+1}$ and $e_{i+1,i+2}$ do not span a triangle except
$F_{n-1,n,n+1}$ and $F_{2n-1,2n,1}$ as the edge sets
$\{e_{n-1,n},e_{n,n+1}\}$ and $\{e_{2n-1,2n},e_{2n,1}\}$ span
triangles in the triangular ladder graph $L^\ast_n$. Moreover,
$F_{i,i+1,j}\in \Omega(L^\ast_n)$ for indices of types $1\leq i\leq
n-1$; $j=2n+1-i$ and $1\leq i\leq n-2$; $j=2n-1-i$. Also,
$F_{i,j,j+1}\in \Omega(L^\ast_n)$ for indices of types $1\leq i\leq
n-2$; $j=2n-1-i$ and $2\leq i\leq n-1$; $j=2n+1-i$. Hence the
result.
\end{proof}
\begin{figure}[htbp]
        \centerline{\includegraphics[width=12.0cm]{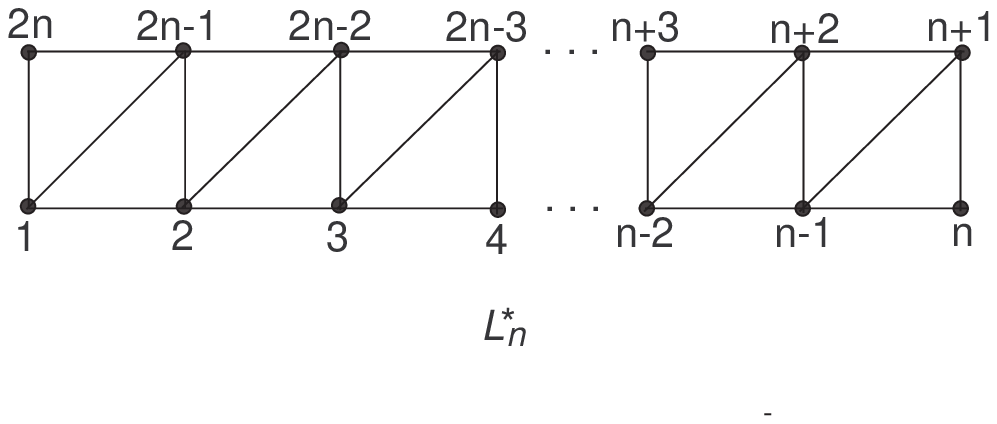}}
        \label{fig Ex5}
\end{figure}
\begin{Theorem}\label{t1}
Let $\Delta_\Gamma(L^\ast_n)$ be the Gallai simplicial complex of
triangular ladder graph $L^\ast_n$ with $2n$ vertices for $n\geq 3$.
Then, the Euler characteristic of $\Delta_\Gamma(L^\ast_n)$ is
$$\chi(\Delta_\Gamma(L^\ast_n))=\sum\limits_{k=0}^N(-1)^k{f_k}=0.$$
\end{Theorem}
\begin{proof}
Since, the triangular ladder graph has $2n$ vertices therefore, we
have $f_0=2n$.\\
Moreover, for $\{l,j,k\}\in \Delta_\Gamma(L^\ast_n)$ with $1\leq
l\leq 2n-2$ and $j,k\in [2n]$, we have
\begin{enumerate}
\item $|\{1,j,k\}|=4$ with $\{j,k\}\in
\{\{2,3\},\{2,2n-2\},\{2,2n\},\{2n-2,2n-1\}\}$;
\item $|\{l,j,k\}|=5(n-3)$ for $2\leq l\leq n-2$ and
$\{j,k\}\in
\{\{l+1,l+2\},\{l+1,2n-1-l\},\{l+1,2n+1-l\},\{2n-1-l,2n-l\},\{2n+1-l,2n+2-l\}\}$;
\item $|\{n-1,j,k\}|=2$ with $\{j,k\}\in \{\{n,n+2\},\{n+2,n+3\}\}$;
\item $|\{l,l+1,l+2\}|=n-1$ for $n\leq l\leq 2n-2$.\\
Adding the results from $(1)$ to $(4)$, we get
$$|\{l,j,k\}|=4+5(n-3)+2+(n-1)=6n-10$$
with $1\leq l\leq 2n-2$ and $j,k\in [2n]$. Therefore, $f_2=6n-10$.\\
Now, for $\{j,k\}\in \Delta_\Gamma(L^\ast_n)$ with $1\leq j\leq
2n-1$ and $k\in [2n]$, we have
  \item $|\{1,k\}|=5$, where $k\in \{2,3,2n-2,2n-1,2n\}$;
  \item $|\{j,k\}|=6(n-3)$ with $2\leq j\leq n-2$ and $k\in
  \{j+1,j+2,2n-1-j,2n-j,2n+1-j,2n+2-j\}$;
  \item $|\{n-1,k\}|=4$, where $k\in \{n,n+1,n+2,n+3\}$;
  \item $|\{j,k\}|=2(n-1)$ with $n\leq j\leq 2n-2$ and $k\in
  \{j+1,j+2\}$;
  \item$|\{2n-1,2n\}|=1$.\\
Adding the results from $(5)$ to $(9)$, we obtain
$$|\{j,k\}|=5+6(n-3)+4+2(n-1)+1=8n-10,$$
where $1\leq j\leq 2n-1$ and $k\in [2n]$. Therefore, $f_1=8n-10$.
\end{enumerate}
Thus, we compute
$$\chi(\Delta_\Gamma(L^\ast_n))=f_0-f_1+f_2=2n-(8n-10)+(6n-10)=0,$$
which is the desired result.
\end{proof}

The Prism graph $Y_{3,n}$ is a simple graph defined by the cartesian
product $Y_{3,n}=C_3\times P_n$ with $3n$ vertices and $3(2n-1)$
edges. We label the edge-set of $Y_{3,n}$ in the following way;\\
$E(Y_{3,n})=\{e_{1,2},
e_{2,3},e_{3,1},e_{4,5},e_{5,6},e_{6,4},\ldots,e_{3i+1,3i+2},e_{3i+2,3i+3},e_{3i+3,3i+1},\ldots,\\
e_{3n-2,3n-1},e_{3n-1,3n},e_{3n,3n-2},e_{1,4},e_{4,7},\ldots,e_{3n-5,3n-2},e_{2,5},e_{5,8},\ldots,e_{3n-4,3n-1},e_{3,6},\\
e_{6,9},\ldots,e_{3n-3,3n}\}$, where
$e_{3i+1,3i+2},e_{3i+2,3i+3},e_{3i+3,3i+1}$ for $0\leq i\leq n-1$
are the edges of $(i+1)$-$th$ $C_3$-cycle.
\begin{Lemma}\label{l2}
Let $Y_{3,n}$ be a prism graph on the vertex set $[3n]$ and edge set $E(Y_{3,n})$, with labeling of edges
given above. Then, we have\\
$\Omega(Y_{3,n})=\{F_{1,2,4},F_{1,2,5},F_{2,3,5},F_{2,3,6},
F_{4,5,1},F_{4,5,2},F_{4,5,7},F_{4,5,8},F_{5,6,2},
F_{5,6,3},F_{5,6,8},\\F_{5,6,9},\ldots,
F_{3n-5,3n-4,3n-8},F_{3n-5,3n-4,3n-7},F_{3n-5,3n-4,3n-2},
F_{3n-5,3n-4,3n-1},\\F_{3n-4,3n-3,3n-7},F_{3n-4,3n-3,3n-6},F_{3n-4,3n-3,3n-1},F_{3n-4,3n-3,3n},
F_{3n-2,3n-1,3n-5},\\F_{3n-2,3n-1,3n-4},F_{3n-1,3n,3n-4},F_{3n-1,3n,3n-3},
F_{3,1,6},F_{3,1,4},F_{6,4,3}, F_{6,4,1},F_{6,4,9},F_{6,4,7},\\
\ldots,F_{3n-3,3n-5,3n-6},F_{3n-3,3n-5,3n-8},F_{3n-3,3n-5,3n},
F_{3n-3,3n-5,3n-2}, F_{3n,3n-2,3n-3},\\F_{3n,3n-2,3n-5},
F_{1,4,7},\ldots,F_{3n-8,3n-5,3n-2},F_{2,5,8},
\ldots,F_{3n-7,3n-4,3n-1}, F_{3,6,9},\\
\ldots, F_{3n-6,3n-3,3n}\}$.
\end{Lemma}
\begin{proof}
By definition, one can easily see that $F_{3i+1,3i+2,3i+3}$ does not
belong to $\Omega(Y_{3,n})$ because $3i+1,3i+2,3i+3$ with $0\leq
i\leq n-1$ are vertices of $(i+1)$-$th$ $C_3$-cycle. Therefore, from
construction of all possible triangles in
prism graph $Y_{3,n}$, we have\\
$(i)$ $F_{j,j+1,j-3},F_{j,j+1,j-2}\in \Omega(Y_{3,n})$ for $4\leq
j\leq
  3n-1$ but $j$ is not multiple of 3;\\
$(ii)$ $F_{j,j+1,j+3},F_{j,j+1,j+4}\in \Omega(Y_{3,n})$ for $1\leq
j\leq
  3n-4$ but $j$ is not multiple of 3;\\
$(iii)$ $F_{3j,3j-2,3j-3},F_{3j,3j-2,3j-5}\in \Omega(Y_{3,n})$ for
$2\leq
j\leq n$;\\
$(iv)$ $F_{3j,3j-2,3j+3},F_{3j,3j-2,3j+1}\in \Omega(Y_{3,n})$ for
$1\leq
j\leq n-1$;\\
$(v)$ $F_{j,j+3,j+6}\in \Omega(Y_{3,n})$ for $1\leq j\leq
  3n-6$.\\
Hence the proof.
\end{proof}
\begin{figure}[htbp]
        \centerline{\includegraphics[width=12.0cm]{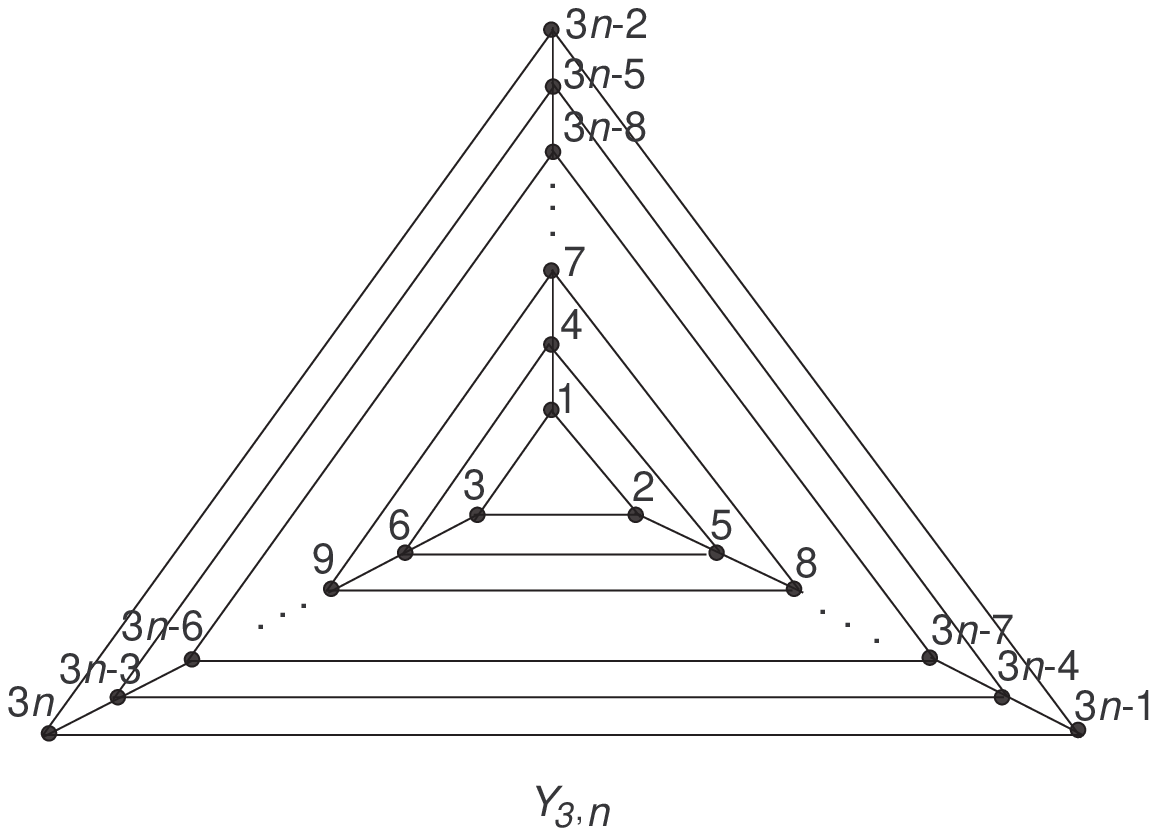}}
        \label{fig Ex5}
\end{figure}

\begin{Theorem}\label{t3}
Let $\Delta_\Gamma(Y_{3,n})$ be the Gallai-simplicial complex of
prism graph $Y_{3,n}$ with $3n$ vertices for $n\geq 3$. Then, the
Euler characteristic of $\Delta_\Gamma(Y_{3,n})$ is
$$\chi(\Delta_\Gamma(Y_{3,n}))=\sum\limits_{k=0}^N(-1)^k{f_k}=3(n-1).$$
\end{Theorem}
\begin{proof}
Since, the prism graph has $3n$ vertices therefore, we have
$f_0=3n$.\\
Now, for $\{3l+i,j,k\}\in \Delta_\Gamma(Y_{3,n})$ with $0\leq l\leq
n-2$ and $j,k\in[3n]$ such that $i=1,2,3$, we have
\begin{enumerate}
  \item $|\{3l+1,j,k\}|=7(n-2)$ with $0\leq l\leq n-3$ and
  $\{j,k\}\in
  \{\{3l+2,3l+4\},\{3l+2,3l+5\},\{3l+3,3l+4\},\{3l+3,3l+6\},\{3l+4,3l+5\},\{3l+4,3l+6\},\{3l+4,3l+7\}\}$;
  \item $|\{3l+2,j,k\}|=5(n-2)$ for $0\leq l\leq n-3$ and
  $\{j,k\}\in
  \{\{3l+3,3l+5\},\{3l+3,3l+6\},\{3l+4,3l+5\},\{3l+5,3l+6\},\{3l+5,3l+8\}\}$;
  \item $|\{3l+3,j,k\}|=3(n-2)$ for $0\leq l\leq n-3$ and $\{j,k\}\in
  \{\{3l+4,3l+6\},\{3l+5,3l+6\},\{3l+6,3l+9\}\}$;
  \item $|\{3n-5,j,k\}|=6$, where $\{j,k\}\in
  \{\{3n-4,3n-2\},\{3n-4,3n-1\},\{3n-3,3n-2\},\{3n-3,3n\},\{3n-2,3n-1\},\{3n-2,3n\}\}$;
  \item $|\{3n-4,j,k\}|=4$, where $\{j,k\}\in
  \{\{3n-3,3n-1\},\{3n-3,3n\},\{3n-2,3n-1\},\{3n-1,3n\}\}$;
  \item $|\{3n-3,j,k\}|=2$, where $\{j,k\}\in
  \{\{3n-2,3n\},\{3n-1,3n\}\}$.\\
Adding the results from $(1)$ to $(6)$, we get
$$f_2=7(n-2)+5(n-2)+3(n-2)+6+4+2=15n-18.$$
Next, for $\{3j+i,k\}\in \Delta_\Gamma(Y_{3,n})$ with $0\leq j\leq
n-2$ and $k\in[3n]$ such that $i=1,2,3$, we obtain
  \item $|\{3j+1,k\}|=6(n-2)$ with $0\leq j\leq n-3$ and
  $k\in \{3j+2,3j+3,3j+4,3j+5,3j+6,3j+7\}$;
  \item $|\{3j+2,k\}|=5(n-2)$ with $0\leq j\leq n-3$ and $k\in
  \{3j+3,3j+4,3j+5,3j+6,3j+8\}$;
  \item $|\{3j+3,k\}|=4(n-2)$ with $0\leq j\leq n-3$ and $k\in
  \{3j+4,3j+5,3j+6,3j+9\}$;
  \item $|\{3n-5,k\}|=5$, where $k\in \{3n-4,3n-3,3n-2,3n-1,3n\}$;
  \item $|\{3n-4,k\}|=4$, where $k\in \{3n-3,3n-2,3n-1,3n\}$;
  \item $|\{3n-3,k\}|=3$, where $k\in \{3n-2,3n-1,3n\}$.\\
Moreover, we have
  \item $|\{3n-2,k\}|=2$, where $k\in \{3n-1,3n\}$;
  \item $|\{3n-1,3n\}|=1.$\\
Adding the results from $(7)$ to $(14)$, we get
$$f_1=6(n-2)+5(n-2)+4(n-2)+5+4+3+2+1=15n-15.$$
\end{enumerate}
Hence, we compute
$$\chi(\Delta_\Gamma(Y_{3,n}))=f_0-f_1+f_2=3n-(15n-15)+(15n-18)=3(n-1),$$
which is the desired result.
\end{proof}

\section{Construction of $f$-Gallai graphs}

We introduce first the $f$-Gallai graph.

\begin{Definition}
{\rm A finite simple graph $G$ is said to be $f$-Gallai graph, if
the edge ideal $I(\Gamma(G))$ of the Gallai graph $\Gamma(G)$ is an
$f$-ideal}.
\end{Definition}
The following theorem provided us a construction of $f$-graphs.

\begin{Theorem} \cite{MAZ}.\label{th1}
{\rm Let $G$ be a simple graph on $n$ vertices. Then for the
following
constructions, $G$ will be $f$-graph:\\
Case(i) When $n= 4l$. $G$ consists of two components $G_1$ and $G_2$
joined with $l$-edges, where both $G_1$ and $G_2$ are the complete
graphs on $2l$ vertices.\\
Case(ii) When $n= 4l+ 1$. $G$ consists of two components $G_1$ and
$G_2$ joined with $l$-edges, where $G_1$ is the complete graph on
$2l$ vertices and $G_2$ is the complete graph on $2l+ 1$ vertices.}
\end{Theorem}
\begin{Definition}
{\rm The star graph $S_n$ is a complete bipartite graph $K_{1,n}$ on
$n+1$ vertices and $n$ edges formed by connecting a single vertex
(central vertex) to all other vertices.}
\end{Definition}

We establish now the following result.

\begin{Theorem}\label{t5}
{\rm Let $G$ be a finite simple graph on $n$ vertices of the form $n=3l+2$ or $3l+3$. Then for the following constructions, $G$ will
be $f$-Gallai graph.\\
{\bf Type 1.}  When $n=3l+2$. $G=\mathbb{S}_{4l}$ is a graph
consisting of two copies of star graphs $S_{2l}$ and $S'_{2l}$ with
$l\geq 2$ having $l$ common vertices.\\ {\bf Type 2.} When $n=3l+3$.
$G=\mathbb{S}_{4l+1}$ is a graph consisting of two star graphs
$S_{2l}$ and $S_{2l+1}$ with $l\geq 2$ having $l$ common vertices.}
\end{Theorem}
\begin{proof}
{\bf Type 1.} When $n=3l+2$, the number of edges in
$\mathbb{S}_{4l}$ will be $4l$, as shown in figure $\mathbb{S}_{12}$
with l=3. Let $\{e_1,\ldots,e_{2l}\}$ and $\{e'_1,\ldots,e'_{2l}\}$
be the edge sets of the star graphs $S_{2l}$ and $S'_{2l}$,
respectively such that $e_i$ and $e'_i$ have a common vertex for
each $i=1,\ldots,l$. While finding Gallai graph
$\Gamma(\mathbb{S}_{4l})$ of the graph $\mathbb{S}_{4l}$, we observe
that the edges $e_1,\ldots,e_{2l}$ of the star graph $S_{2l}$ in
$\mathbb{S}_{4l}$ will induce a complete graph $\Gamma(S_{2l})$ on
$2l$ vertices in the Gallai graph $\Gamma(\mathbb{S}_{4l})$, as
shown in figure $\Gamma(\mathbb{S}_{12})$ with $l=3$. Similarly, the
edges $e'_1,\ldots,e'_{2l}$ of the star graph $S'_{2l}$ will induce
another complete graph $\Gamma(S'_{2l})$ on $2l$ vertices in
$\Gamma(\mathbb{S}_{4l})$. As, $e_i$ and $e'_i$ are the adjacent
edges in $\mathbb{S}_{4l}$ for each $i=1,\ldots,l$. Therefore, $e_i$
and $e'_i$ will be incident vertices in $\Gamma(\mathbb{S}_{4l})$
for every $i=1,\ldots,l$. Thus, Gallai graph
$\Gamma(\mathbb{S}_{4l})$ having $4l$ vertices consists of two
components $\Gamma(S_{2l})$ and $\Gamma(S'_{2l})$ joined with $l$-
edges, where both $\Gamma(S_{2l})$ and $\Gamma(S'_{2l})$ are
complete graphs on $2l$ vertices. Therefore, by Theorem \ref{th1},
$\Gamma(\mathbb{S}_{4l})$ is $f$-Gallai graph.

\begin{figure}[htbp]
        \centerline{\includegraphics[width=14.0cm]{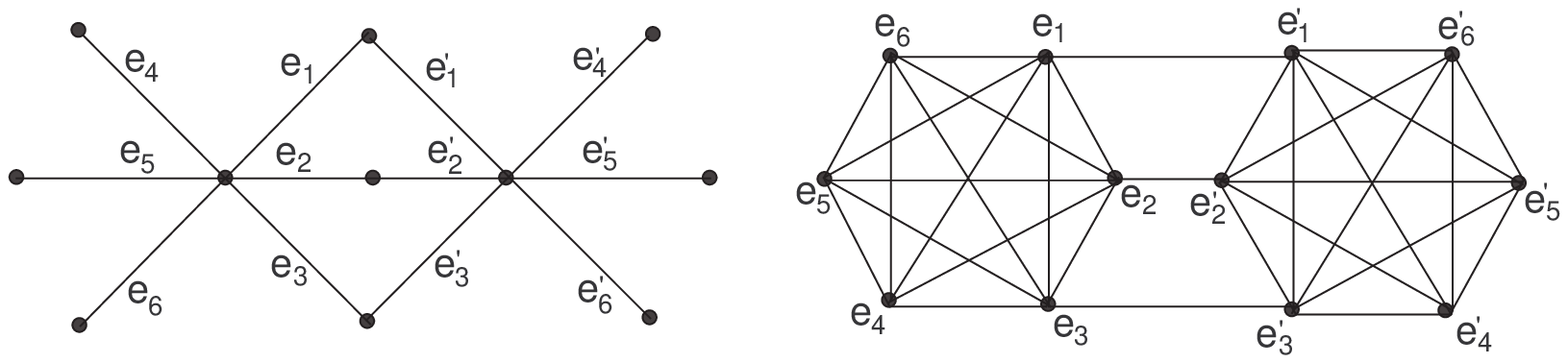}}
        \label{fig Ex5}
        $\mathbb{S}_{12}$ \ \ \ \ \ \ \ \ \ \ \ \ \ \ \ \ \ \ \ \ \ \ \ \ \ \ \ \ \ \ \ \ \ \ \ \ \ \ \ \ \ \ \  $\Gamma(\mathbb{S}_{12})$
\end{figure}
{\bf Type 2.} When $n=3l+3$, the number of edges in
$\mathbb{S}_{4l+1}$ will be $4l+1$, see figure $\mathbb{S}_{13}$
(where $l=3$). Let $\{e_1,\ldots,e_{2l}\}$ and
$\{e'_1,\ldots,e'_{2l+1}\}$ be the edge sets of the star graphs
$S_{2l}$ and $S_{2l+1}$ (respectively) such that $e_i$ and $e'_i$
share a common vertex for each $i=1,\ldots,l$. One can easily see
that the edges $e_1,\ldots,e_{2l}$ of $S_{2l}$ in
$\mathbb{S}_{4l+1}$ will induce a complete graph $\Gamma(S_{2l})$ on
$2l$ vertices in the Gallai graph $\Gamma(\mathbb{S}_{4l+1})$, see
figure $\Gamma(\mathbb{S}_{13})$ (where $l=3$). Similarly, the edges
$e'_1,\ldots,e'_{2l+1}$ of $S_{2l+1}$ will induce another complete
graph $\Gamma(S_{2l+1})$ on $2l+1$ vertices in
$\Gamma(\mathbb{S}_{4l+1})$. Since $e_i$ and $e'_i$ are the adjacent
edges in $\mathbb{S}_{4l+1}$ for every $i=1,\ldots,l$. Therefore,
$e_i$ and $e'_i$ will be incident vertices in the Gallai graph
$\Gamma(\mathbb{S}_{4l+1})$ for each $i=1,\ldots,l$. Thus, the
Gallai graph $\Gamma(\mathbb{S}_{4l+1})$ having $4l+1$ vertices
consists of two components $\Gamma(S_{2l})$ and $\Gamma(S_{2l+1})$
joined with $l$-edges, where $\Gamma(S_{2l})$ and $\Gamma(S_{2l+1})$
are complete graphs on $2l$ and $2l+1$ vertices, respectively.
Hence, by Theorem \ref{th1}, $\Gamma(\mathbb{S}_{4l+1})$ is
$f$-Gallai graph.

\begin{figure}[htbp]
        \centerline{\includegraphics[width=14.0cm]{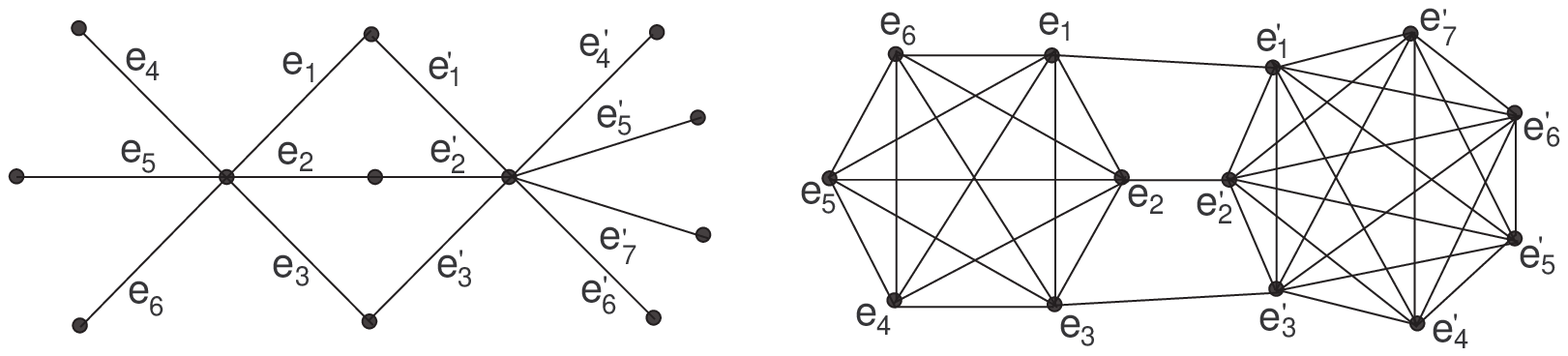}}
        \label{fig Ex5}
        $\mathbb{S}_{13}$ \ \ \ \ \ \ \ \ \ \ \ \ \ \ \ \ \ \ \ \ \ \ \ \ \ \ \ \ \ \ \ \ \ \ \ \ \ \ \ \ \ \ \  $\Gamma(\mathbb{S}_{13})$
\end{figure}
\end{proof}

\begin{Remark} One can easily see that the Gallai graph of the
line graph $L_n$ is isomorphic to $L_{n-1}$ and that of cyclic graph
$C_n$ is isomorphic to $C_n$. Therefore, both $\Gamma(L_n)$ and
$\Gamma(C_n)$ are $f$-Gallai graphs if and only if $n=5$, see
\cite{MAZ}.
\end{Remark}


\begin{thebibliography}{1}

\bibitem{AAAB} G. Q. Abbasi, S. Ahmad, I. Anwar, W. A. Baig, {\it $f$-ideals of degree
2}, Algebra Colloquium, 19 (2012), no. 1, 921-926.

\bibitem{AKNS} I. Anwar, Z. Kosar and S. Nazir, {\it An Efficient Algebraic Criterion For Shellability}, arXiv: 1705.09537.

\bibitem{AMBZ} I. Anwar, H. Mahmood, M. A. Binyamin and M. K. Zafar, {\it On the Characterization of
f-Ideals}, Communications in Algebra 42 (2014), no. 9, 3736-3741.




\bibitem{WBH} W. Bruns and J. Herzog, {\it Cohen-Macaulay Rings}, Revised
Edition, Cambridge Studies in Advanced Marthematics, Vol. 39,
Cambridge University Press, Cambridge, 1998.

\bibitem{SF} S. Faridi, {\it The Facet Ideal of a Simplicial
Complex}, Manuscripta Mathematica, 109 (2002), 159-174.

\bibitem{TG} T. Gallai, {\it Transitiv Orientierbare Graphen}, Acta Math.
Acad. Sci. Hung., 18 (1967), 25-66.

\bibitem{H} A. Hatcher, {\it Algebraic Topology}, Cambridge
University Press, 2002.



\bibitem{VBL} V. B. Le, {\it Gallai Graphs and Anti-Gallai Graphs}, Discrete
Math., 159 (1996), 179-189.

\bibitem{MAZ} H. Mahmood, I. Anwar, M. K. Zafar, {\it A Construction of Cohen-Macaulay $f$-Graphs}, Journal of Algebra and its
Applications, 13 (2014), no. 6, 1450012-1450019.

\bibitem{M} W.S. Massey, {\it Algebraic Topology, An Introduction},
Springer-Verlag, New York, 1977.


\end{thebibliography}
\end{document}